%% file: Standard-FD2-25.tex
\documentclass[12pt]{amsart}
\usepackage{amscd,amssymb,epsfig,mathtools}
\calclayout

\numberwithin{equation}{section} 

\newcommand{\ud}{\,d} 
 
\newcommand{\R}{\mathbb{R}}

\newcommand\cof{\operatorname{cof}}

\newcommand\tr{\operatorname{tr}}

\newcommand\sym{\operatorname{sym}}

\renewcommand{\div}{\operatorname{div}}

\newcommand{\tir}[1]{\ensuremath{\overline {#1}}} 
\newtheorem{thm}{Theorem}[section] 
 
\newtheorem{lemma}[thm]{Lemma} 
 
\newtheorem{prop}[thm]{Proposition} 
 
\newtheorem{defn}[thm]{Definition} 
 
\newtheorem{rem}[thm]{Remark}

\def\whsq{\vbox to 5.8pt 
{\offinterlineskip\hrule 
\hbox to 5.8pt{\vrule height 
5.1pt\hss\vrule height 5.1pt}\hrule}}

\def\<{\langle} 
\def\>{\rangle} 
\def\PP{{\mathop{{\rm I}\kern-.2em{\rm P}}\nolimits}} 
\def\FF{{\mathop{{\rm I}\kern-.2em{\rm F}}\nolimits}}   
\def\ZZ{{\mathop{{\rm I}\kern-.2em{\rm Z}}\nolimits}} 
 
\newlength{\sidemargin} 
\begin{document}

\title[]{
On standard finite difference discretizations of the elliptic Monge-Amp\`ere equation}

\author{Gerard Awanou}
\address{Department of Mathematics, Statistics, and Computer Science, M/C 249.
University of Illinois at Chicago, 
Chicago, IL 60607-7045, USA}
\email{awanou@uic.edu}  
\urladdr{http://www.math.uic.edu/\~{}awanou}

\maketitle

\begin{abstract}
Given an orthogonal lattice with mesh length $h$ on a bounded convex domain $\Omega$, we 
propose to approximate the Aleksandrov solution of the Monge-Amp\`ere equation by regularizing the data and discretizing the equation in a subdomain using the standard finite difference method. The Dirichlet data is used
to approximate the solution in the remaining part of the domain. We prove the uniform convergence on compact subsets of the solution of the discrete problems to an approximate problem on the subdomain. 
The result explains the behavior of methods based on the standard finite difference method and designed to numerically converge to non-smooth solutions. We propose an algorithm which for smooth solutions appears faster than the popular Newton's method with a high accuracy for 
non smooth solutions. 
The convergence of the algorithm is independent of how close to the numerical solution the initial guess is, upon rescaling the equation and given a user's measure of the closeness of an initial guess. 
\end{abstract}

\input{Standard-FD2-25a}

\input{Standard-FD2-25b}

\section*{Acknowledgements}
This work began when the author was supported in part by a 2009-2013 Sloan Foundation Fellowship and continued while the author was in residence at
the Mathematical Sciences Research Institute (MSRI) in Berkeley, California, Fall 2013. The MSRI receives major funding from the National Science
Foundation under Grant No. 0932078 000. The author was partially supported by NSF DMS grant No 1319640.


\end{document}

%% file: Standard-FD2-25a.tex
\section{Introduction}
Let $\Omega$ be  a bounded convex domain of $\R^d, d \geq 2$ and let 
$g \in C(\partial \Omega)$, 
$f \in C(\Omega)$ with $0 < c_0 \leq  f \leq c_1$ for constants  $c_0, c_1 \in \R$. 
We assume that $g \in C(\partial \Omega)$ can be extended to a function $\tilde{g} \in C(\tir{\Omega})$ which is  convex in $\Omega$.
We are interested in the finite difference approximation of the Aleksandrov solution of the Monge-Amp\`ere equation
\begin{equation} \label{m1}
\det D^2 u =f \, \text{in} \, \Omega,  u =g \, \text{on} \, \partial \Omega.
\end{equation}
To a convex function $u$, one associates a measure $M[u]$ and \eqref{m1} is said to have an Aleksandrov solution if the density of $M[u]$ with respect to the Lebesgue measure is $f$. If $u \in C^2 (\Omega)$, $M[u]$ is a measure with density  $\det D^2 u$, where $D^2 u=\bigg( (\partial^2 u) / (\partial x_i \partial x_j)\bigg)_{i,j=1,\ldots, d}  $ 
is the Hessian of $u$. 
There are several equivalent definitions of the Monge-Amp\`ere measure in the general case and the simplest approach is to use an analytic definition based on approximation by smooth functions. See section \ref{Aleks-def} and \cite{Rauch77} for the equivalent definitions. 

We propose to approximate the Aleksandrov solution of \eqref{m1} by regularizing the data and discretizing the equation in a subdomain using the standard finite difference method. The Dirichlet data is used
to approximate the solution in the remaining part of the domain. We prove the uniform convergence on compact subsets of the solution of the discrete problems to an approximate problem on the subdomain. 

Our numerical algorithms are of the time marching types with proven convergence. The time marching method, with the central finite difference discretization previously used in \cite{Benamou2010,Awanou-k-Hessian12}, appears faster than Newton's method {\it for smooth solutions}. In some cases 15 times faster. It is shown to be numerically robust for non smooth solutions of the Monge-Amp\`ere equation with right hand side absolutely continuous with respect to the Lebesgue measure. The convergence is shown to be independent of the closeness of an initial guess upon rescaling the equation. We introduce a compatible discretization in the sense that it reproduces at the continuous level essential features of the continuous problem. We prove for smooth solutions an asymptotic convergence rate  of the discretization. We observed that when the time marching method  for the central discretization is used to provide a starting point for the time marching method for the compatible discretization, one reaches a high accuracy for non smooth solutions.

\subsection{Methodology for smooth solutions}
We introduce a compatible discretization which allows us to give a proof of convergence of the discretization for smooth solutions and a proof of an asymptotic convergence rate similar to the proofs for the finite element discretization of \eqref{m1}, c.f. \cite{Awanou-Std01} and the references therein. Two key ideas used in this paper, which were not used in the finite element papers \cite{Bohmer2008,Feng2009,Brenner2010b}, are the use of the continuity of the eigenvalues of a matrix as a function of its entries and rescaling the equation.  

\subsection{Methodology for non smooth solutions} \label{methodns} 
We regularize the data by considering functions $f_m, g_m \in C^{\infty}(\tir{\Omega})$ such that $0 < c_2 \leq f_m \leq c_3 $, $f_m$ converges uniformly to $f$ on $\tir{\Omega}$  and $g_m$ converges uniformly to $\tilde{g}$ on $\tir{\Omega}$. See \cite{MongeC1Alex} for an example. The second key idea of this paper is to consider a sequence of smooth uniformly convex subdomains $\Omega_s$ which converges to $\Omega$ \cite{Blocki97}.

We consider in this paper ''interior'' discretizations. By this, we mean that we prove convergence of the discretization in an interior domain. Values at mesh points closest to the boundary are approximated using the boundary values. 
Let $ \delta >0$ be a small parameter. 
We will need a theoretical computational domain $\widetilde{\Omega}$ chosen as a subdomain of $\Omega$. 
We require that
$$
 \widetilde{\Omega} \subset \Omega_s, \ \text{for all} \ s.
$$

It is known, c.f. \cite{Awanou-Std04} or \cite[Proposition 2.4 ]{Savin13}, that the  Aleksandrov solution of
\begin{equation}
\det  D^2 u_m = f_m  \ \text{in} \, \Omega, \, u_m=g_m \, \text{on} \, \partial \Omega, 
\label{m1-m}
\end{equation}
converges uniformly on compact subsets of $\Omega$ to the  Aleksandrov solution $u$ of \eqref{m1}.
 
We choose $\tilde{m}$ 
such that $| f(x) - f_{\tilde{m}}(x)| < \delta$, $| g(x) - g_{\tilde{m}}(x)| < \delta$ and $| u(x) - u_{\tilde{m}}(x)| < \delta$ for all $x \in \tir{\Omega}$.

We show in this paper that given a mesh on $\Omega$, an ''interior''  discretization (c.f. \eqref{m1h-altX} and \eqref{m1hX} below) of the problem
\begin{equation} \label{m11}
\det D^2 u_{\tilde{m} s} =f_{\tilde{m}} \, \text{in} \, \widetilde{\Omega},  u_{\tilde{m} s} =u_{\tilde{m}} \, \text{on} \, \partial \widetilde{\Omega}.
\end{equation}
has a unique local solution $u_{\tilde{m} s,h}$ which is a discrete convex function. 
We discretize the Hessian using the standard finite difference method and show that the solution $u_{h}$ of the resulting discrete problem, is the limit of a subsequence in 
$s$ of $u_{\tilde{m} s,h}$ where $u_{\tilde{m} s,h}$ is the finite difference approximation of the solution $u_{\tilde{m} s}$ of \eqref{m11}. We prove that $u_h$ converges uniformly on compact subsets of $\widetilde{\Omega}$ to the solution $\tilde{u}$
of 
\begin{equation}
\det  D^2 \tilde{u} = f_{\tilde{m}}  \ \text{in} \, \widetilde{\Omega}, \, \tilde{u}=u_{\tilde{m}} \, \text{on} \, \partial \widetilde{\Omega}. 
\label{m1m}
\end{equation}
The solution $u$ of \eqref{m1} can then be approximated within a prescribed accuracy by first choosing $\tilde{m}$ and then $h$ sufficiently small. We emphasize that the solution $\tilde{u}$ of \eqref{m1m} is not necessarily smooth.

A technical aspect of the proof is that we use {\it interior} second order derivative estimates of the solution $u_{ms}$ as the latter may blow up on the boundary $\partial \Omega$ if the latter is not strictly convex. 
 As a consequence of the interior Schauder estimates, we obtain stability on compact subsets of $\widetilde{\Omega}$ of  the discretization. This is one of the  main contributions of the paper and is treated in section \ref{cvg}.

For simplicity, the dependence of $\tilde{u}$ on $\tilde{m}$ is not indicated. By unicity of the Aleksandrov solution $u_m$ of \eqref{m1-m}, 
we have $\tilde{u}=u_{\tilde{m}}$ in $\widetilde{\Omega}$ and hence as $\widetilde{\Omega} \to \Omega$, $u_{\tilde{m}}|_{\partial \widetilde{\Omega}} \to g|_{\partial \Omega}$.   Thus, from a practical point of view, for the implementation, we see that one can take $\widetilde{\Omega}=\Omega$, $f_{m}=f$ with $u_h=g$ on $\partial \Omega$. It is in that sense that the results of this paper explains the behavior of methods based on the standard finite difference method and designed to numerically converge to non-smooth solutions. 


\subsection{Significance of the results in relation with other work}

A proven convergence proof for Aleksandrov solutions was given for the two dimensional problem for the discretization proposed in \cite{Oliker1988}. The approach through the so-called viscosity solutions was considered in \cite{Oberman2010a} in the context of monotone finite difference schemes. 
We classify these methods as nonstandard. 
While the discretization proposed in \cite{Froese13} uses a standard discretization in parts of the domain, it is still a nonstandard discretization as it uses a monotone scheme in parts of the domain.

The central finite difference discretizations is popular in science and engineering \cite{Headrick05,Chen2010b,Chen2010c}. 
However solving the resulting nonlinear discrete system of equations  by Newton's method produces disastrous results when \eqref{m1} has a non smooth solution. Ever since the pioneering work \cite{Dean2004}, various approaches have been proposed to solve the nonlinear equations for convergence to non smooth solutions e.g. \cite{Benamou2010,Mohammadi2007}.  Despite the efficiency of the method proposed in \cite{Mohammadi2007}, the {\it fundamental} question of a proof of convergence of the discretization had not been solved. 

In \cite{Awanou-k-Hessian12}, for smooth solutions, we proved the quadratic convergence rate of the central finite difference discretization and the convergence of Newton's method for solving the resulting nonlinear system of equations.  The quadratic convergence rate of the central finite difference discretization for smooth solutions was only known as `` formally second-order accurate''  \cite{Benamou2010}. 

One can make an analogy between numerical methods for the Monge-Amp\`ere equation and the setting of the elementary Newton's method for solving a nonlinear equation $p(x)=0$ with multiple real roots. It is well known that this is an efficient method depending on the initial guess. The situation is better for the Monge-Amp\`ere equation as one is interested in a certain kind of numerical solutions and the equation can be rescaled making the various possible solutions far from each other. Thus the results of this paper do not contradict the observations made in \cite[Section 1]{Feng2013}.

The issues pertinent to the analysis of standard discretizations can be summarized as follows.
\begin{enumerate}
\item Prove the existence and uniqueness of a solution to the discrete problem
\item Prove the convergence of the discretization
\item Prove the convergence of an iterative method for solving the discrete problem.
\end{enumerate}

We address all three issues completely in this paper. The distinguished feature of the methods discussed in this paper, like the ones discussed in  \cite{Benamou2010}, is to preserve weakly convexity in the iterations, c.f. Remark \ref{subh-rem}. In the iterations, the positivity of the discrete Laplacian is preserved. This feature allows the processes to avoid spurious solutions. Our numerical experiments add to the growing evidence of the effectiveness of the standard finite difference method for the Monge-Amp\`ere equation \cite{Benamou2010,Oberman2010b,Froese13}. In both \cite{Oberman2010b,Froese13}, the effectiveness of the standard finite difference method for smooth solutions was taken advantage of. Since the fast convergence of Newton's method for smooth solutions was one of the motivations behind the discretizations proposed in \cite{Froese13,Oberman2010a,Oberman2010b}, it is reasonable to expect that the time marching method would be faster than these methods, at least for smooth solutions.

We use an abstract treatment of the Monge-Amp\`ere measure, see section \ref{Aleks-def}, which shows a close connection with the more natural definition of $\det D^2 u(x)$ for smooth solutions. We believe that this connection, which also encapsulates the geometric structure of the Monge-Amp\`ere equation \cite{Rauch77}, is the prime reason standard discretizations work for Monge-Amp\`ere type equations. 


A standard finite difference discretization of the Dirichlet problem for the Monge-Amp\`ere equation was introduced in \cite{Dean06b}. Finite element discretizations have also been proposed, e.g. \cite{GlowinskiICIAM07,Bohmer2008,Feng2009,Brenner2010b,Lakkis11b,Davydov12,Glowinski2014}. 
Since we use standard discretizations, the efficient tools developed for computational mathematics such as adaptive mesh refinements and multigrid algorithms can be transferred seamlessly to the Monge-Amp\`ere context. 

The lack of a maximum principle for the discretizations analyzed in this paper is related to the difficulty of proving stability of the discretization for smooth solutions without assuming a bound on a high order norm of the solution. For that reason, we introduced the theoretical computational domain $\widetilde{\Omega}$ and fix the parameter $\tilde{m}$ in the regularization of the data.


\subsection{Summary of contributions and broader impacts}

The contributions and broader impacts of this paper are therefore
\begin{enumerate}
\item We present a theory which contributes to the resolution of the long standing open problem of convergence of standard finite difference discretizations to viscosity solutions of the Monge-Amp\`ere equation. 
We recall that the notion of viscosity and Aleksandrov solutions are equivalent for $f >0$ and continuous on $\tir{\Omega}$. 
\item The introduction and proof of convergence of a highly accurate structure preserving finite difference discretization of the Monge-Amp\`ere equation which is of standard type.
\item A proof of an asymptotic convergence rate for a standard compatible finite difference discretization is given in the case of smooth solutions.
\item A ''canonical proof'' is given in the sense that the same principles may be applied to other type of discretizations. It does not seem possible to adapt the non standard discretizations to the finite element context.
\item A convergence proof is given for a time marching algorithm which can be used to solve the discrete problem. It consists of a Laplacian preconditioner of a simple gradient algorithm. 
\item The proof of convergence to the Aleksandrov solutions can be adapted to a wide range of methods once it is understood how these methods solve the discrete problems for smooth solutions. In particular it follows from the general approach taken in this paper that the monotone schemes introduced in \cite{Oberman2010a,Froese13} converge to both viscosity solutions and Aleksandrov solutions. 
For right hand sides which approximate a combination of Dirac masses, a very good initial guess is necessary for these methods.
Nethertheless the result is important for optimal transportation problems where one has to extend the data, resulting in a discontinuous right hand side $f$. In these cases the continuous viscosity solution approach is no longer valid.

\item The 
new point of view of this paper 
places the study of standard finite difference discretizations solely in the framework of traditional numerical analysis. That is, the main task is to understand how schemes perform for smooth solutions and whether they give numerical evidence of convergence for non smooth solutions. 
The convergence of the discretization for non smooth solutions follows from the general approach taken in this paper. This paper thus provides a blueprint which can be used to analyze the discretizations  proposed in \cite{Dean2004,Mohammadi2007}. 
\end{enumerate}
 
\subsection{Organization of the paper}
We organize the paper as follows. In the next section we first introduce standard discretizations of the Monge-Amp\`ere equation and recall the convergence result of \cite{Awanou-k-Hessian12} for smooth solutions and the second order accurate central discretization. In section \ref{gen-frame} we recall  key results on the Aleksandrov theory of the Monge-Amp\`ere equation and give a general framework of convergence of standard discretizations to the Aleksandrov solution. Additional notation and preliminaries are given in section \ref{notation+}.
In section  \ref{time1} we prove the convergence of the compatible discretization for both smooth and non smooth solutions. 
The last section is devoted to a discussion of numerical results for the two dimensional problem using  the standard compatible and central discretizations. The proof of some of the results is given in an appendix. 

\section{Standard finite difference discretizations of the Monge-Amp\`ere equation} \label{std-common}

We recall that $\Omega$ is a bounded convex domain of $\R^d$. For $0<  h < 1$, we define
\begin{align*}
\mathbb{Z}_h & = \{x=(x_1,\ldots,x_d)^T \in \R^d: x_i/h \in \mathbb{Z} \}\\
 \Omega_h & = \tir{\Omega} \cap  \mathbb{Z}_h.
\end{align*}
Let $\mathcal{M}(\Omega_h)$ denote the space of grid functions, i.e. mappings from $\Omega_h$ to $\R$.
We denote by $e_i, i=1,\ldots,d$ the $i$-th unit vector of $\R^d$ and consider first order difference operators defined on $\mathbb{Z}_h$ by 
\begin{align*}
\partial^i_{+} v_h(x) & \coloneqq \frac{ v_h(x+he_i)-v_h(x) } {h}, & 
\partial^i_{-} v_h(x) & \coloneqq \frac{ v_h(x)-v_h(x-he_i) } {h}. 
\end{align*}
We have for $x \in \mathbb{Z}_h$
\begin{align}
\partial^j_{-} \partial^i_{+} v_h(x) & = \frac{v_h(x+he_i)-v_h(x)-v_h(x+he_i-h e_j)+v_h(x-he_j)}{h^2}. \label{second-disc1}
\end{align}
We will also need the central second order accurate first order operator defined for $i=1,\ldots,d$ by
$$
\partial^i_h v_h(x)  \coloneqq \frac{ v_h(x+he_i)-v_h(x-he_i) }{2 h}.
$$

We use the notation $A=(a_{i j})_{i,j=1,\ldots,d}$ to denote the matrix $A$ with entries $a_{ij}$.
Several discrete analogues of the Hessian $D^2 v$ of a $C^2$ function $v$ can be defined for $x \in \mathbb{Z}_h$ and a grid function $v_h$. One possibility is to define the discrete Hessian as the non symmetric matrix field $\mathcal{H}_d(v_h)$ with components
\begin{align*}
(\mathcal{H}_d(v_h)(x))_{i j}  & =\partial^j_{-} \partial^i_{+} v_h(x), i,j=1,\ldots,d.
\end{align*}
\begin{defn}
A $d \times d$ matrix $A$ is said to be positive definite if and only if $z^T A z >0$ for $z \in \R^d, z \neq 0$. The matrix $A$ is said to be positive if and only if $z^T A z \geq 0$ for $z \in \R^d$.
\end{defn}
Decomposing a matrix $A$ into its symmetric and skew symmetric part, i.e.
$
A = (A + A^T)/2 +  (A - A^T)/2,
$
one concludes that $A$ is (positive) definite if and only if its symmetric part is (positive) definite. We will use the notation $\sym A$ to denote the symmetric part of $A$.

Another discretization of the Hessian matrix which has been used in previous work \cite{Benamou2010,Headrick05,Chen2010b,Chen2010c} is to consider for a grid function $v_h$, the matrix field $ \tir{\mathcal{H}}_d(v_h)$  with components
\begin{align*}
(\tir{\mathcal{H}}_d(v_h)(x))_{i i}  &= \partial^i_{+} \partial^i_{-} v_h(x), i,j=1,\ldots,d \\
(\tir{\mathcal{H}}_d(v_h)(x))_{i j}  & = \partial^i_{h} \partial^j_{h} v_h(x), i,j=1,\ldots,d, i \neq j.
\end{align*}

We denote by $\Omega_h^0$ the subset of $\Omega_h$ consisting of grid points $x$ for which  $x \pm h e_i \pm h e_j \in \tir{\Omega}$ for $i,j=1,\ldots,d$ 
and put
$\partial \Omega_h = \Omega_h \setminus \Omega_h^0$. 

For the study of the convergence of numerical methods for non smooth solutions, we will consider the set of interior mesh points
$$
\Omega_h^{00} = \{ \, x \in \widetilde{\Omega} \cap \mathbb{Z}_h, x \pm  2 h e_i \pm 2 h e_j \in \tir{\Omega}, \  \text{for} \ i,j=1,\ldots,d \, \},
$$
and define $\partial \Omega_h^0 = (\tir{\widetilde{\Omega}} \cap \mathbb{Z}_h) \setminus \Omega_h^{00}$. 


The restriction map is defined as a mapping
$$
r_h: C(\Omega) \to \mathcal{M}(\Omega_h), r_h (v) (x) = v(x), x \in \Omega_h,
$$
and is extended canonically to vector fields and matrix fields. 
The restriction to a subset of $\tir{\Omega}$ is defined analogously.

For a vector valued grid function $v_h$ with components $v_{h,i}, i=1,\ldots,d$, the divergence of $v_h$ is defined as the grid function $\div_h v_h = \sum_{i=1}^d \partial^i_- v_{h,i}$. The operator $\div_h$ is extended to matrix fields by taking the divergence of each row. 

We define two discrete versions of the gradient: $D_h v_h$ and $\tir{D}_h v_h$ as: 
\begin{align*}
D_h v_h & : = (\partial^i_+ v_h)_{i=1,\ldots,d}, &
\tir{D}_h v_h & : = (\partial^i_- v_h)_{i=1,\ldots,d}.
\end{align*}

If $v_h=(v_{h,i})_{i=1,\ldots,d}$ is a vector field, we define $\tir{D}_h v_h$ as the matrix field obtained by applying $\tir{D}_h$ to each row, i.e.
$
\tir{D}_h v_h = (\partial^j_- v_{h,i})_{i, j=1,\ldots,d}. 
$
Thus for a scalar field $v_h$
$
\tir{D}_h D_h v_h = \mathcal{H}_d(v_h).
$
The discrete Laplacian $\Delta_h$ is defined as
$
\Delta_h v_h : =  \sum_{i=1}^d \partial^i_+\partial^i_- v_h.
$
With the above definitions, we  have 
$
\div_h D_h v_h = \Delta_h v_h.
$
We first consider two standard discretizations of \eqref{m1}
\begin{align} \label{m1h-alt}
\det  \tir{\mathcal{H}}_d(u_{h}) =r_h(f) \, \text{in} \, \Omega_h^0, u_{h} =r_h(g) \, \text{on} \, \partial \Omega_h,
\end{align}
and
\begin{align} \label{m1h}
\frac{1}{d} \div_h [ (\cof \sym \mathcal{H}_d u_{h} ) D_h u_{h} ]  =r_h(f) \, \text{in} \, \Omega_h^0, u_{h} =r_h(g) \, \text{on} \, \partial \Omega_h.
\end{align}
The latter will be seen as a standard compatible discretization in the sense that essential features of the differential operators at the continuous level are reproduced at the discrete level.

We denote by $\lambda_1(A)$ and $\lambda_d(A)$ the smallest and largest eigenvalues of a symmetric matrix $A$ respectively. 
The discrete analogue of the maximum norm is given 
by \eqref{max-1st} 
We have under smoothness assumptions of the solution $u$ of \eqref{m1}

\begin{thm}[\cite{Awanou-k-Hessian12}] \label{first-cvg}
Problem \eqref{m1h-alt} has a unique local solution $u_h$ with $\lambda_1( \tir{\mathcal{H}}_d (u_h) ) \geq c >0$ for a constant $c$ independent of $h$ and on each compact subset $K$ of $\Omega$
$$
\max_{x \in K} |u_h-r_h u| \leq C h^2,
$$
with a constant $C$ which can be taken as a multiple of $||u||_{C^5(\Omega)}$.
Thus $u_h$ converges uniformly on compact subsets of 
$\Omega$ to the unique smooth convex solution of \eqref{m1}.
\end{thm}

The proof of the following result is given in section \ref{time1}. 

\begin{prop} \label{second-cvg} Problem \eqref{m1h} has a unique local solution $u_h$ with $\lambda_1(\mathcal{H}_d (u_h))\geq c >0$ for a constant $c$ independent of $h$ and
$$
|u_h-r_h u|_{0,\infty,h} \leq C h^2,
$$
with a constant $C$ which can be taken as a multiple of $||u||_{C^2(\Omega)}$.
Thus $u_h$ converges uniformly on 
$\Omega$ to the unique smooth convex solution of \eqref{m1}.

\end{prop}

\begin{defn} \label{def-disc-conv}
A mesh function $v_h$ is said to be discrete convex if $\mathcal{H}_d(v_h)(x)$ (alternatively  $\tir{\mathcal{H}}_d(v_h)(x)$) is a positive matrix for all $x \in \Omega_h^0$. The function $v_h$ is said to be discrete strictly convex if 
$\mathcal{H}_d(v_h)(x)$ (alternatively  $\tir{\mathcal{H}}_d(v_h)(x)$) is a positive definite matrix for all $x \in \Omega_h^0$.
\end{defn}


In \cite{Awanou-k-Hessian12} we proved, for smooth solutions, the local solvability of \eqref{m1h-alt}. In section \ref{time1} we prove the local solvability of \eqref{m1h} for smooth solutions.
We analyze in this paper the convergence of time marching methods for solving respectively \eqref{m1h-alt} and \eqref{m1h} under a smoothness assumption on the solution $u$ of \eqref{m1}. 
They are given by
\begin{align} \label{time-marching}
\begin{split}
- \tir{\nu} \Delta_h u_{h }^{k+1} & = -  \tir{\nu} \Delta_h u_{h }^k + \det \tir{ \mathcal{H}}_d(u_{h}^k) - r_h(f) \, \text{in} \, \Omega_h^0 \\
  u_{h }^{k+1} & = r_h(g) \, \text{on} \, \partial \Omega_h,
\end{split}
\end{align}
and
\begin{align} 
\begin{split}  \label{time-marching3}
-  \tir{\nu} \Delta_h u_{h }^{k+1} & = -  \tir{\nu} \Delta_h u_{h }^k + \frac{1}{d} \div_h [ (\cof \sym \mathcal{H}_d u_{h}^k ) D_h u_{h}^k ]  - r_h(f) \, \text{in} \, \Omega_h^0 \\
  u_{h }^{k+1} & = r_h(g) \, \text{on} \, \partial \Omega_h,
\end{split}
\end{align}
for $ \tir{\nu} >0$ sufficiently large and an initial guess $u_{h}^0$.

\begin{rem} \label{subh-rem}
If one takes $\tir{\nu}$ large in \eqref{time-marching} and \eqref{time-marching3}, one gets that the left hand sides are negative, i.e. discrete subharmonicity is preserved in the iterations.
\end{rem}

For the situation where \eqref{m1} does not have a smooth solution, we consider the related problems
\begin{align} \label{m1h-altX}
\det  \tir{\mathcal{H}}_d(u_{h}) =r_h(f_{\tilde{m}}) \, \text{in} \, \Omega_h^{00}, u_{h} =r_h(u_{\tilde{m}}) \, \text{on} \, \partial \Omega_h^0,
\end{align}
and
\begin{align} \label{m1hX}
\frac{1}{d} \div_h [ (\cof \sym \mathcal{H}_d u_{h} ) D_h u_{h} ]  =r_h(f_{\tilde{m}}) \, \text{in} \, \Omega_h^{00}, u_{h} =r_h(u_{\tilde{m}}) \, \text{on} \, \partial \Omega_h^0,
\end{align}
with corresponding time marching methods
\begin{align} \label{time-marchingX}
\begin{split}
- \tir{\nu} \Delta_h u_{h }^{k+1} & = -  \tir{\nu} \Delta_h u_{h }^k + \det \tir{ \mathcal{H}}_d(u_{h}^k) - r_h(f_{\tilde{m}}) \, \text{in} \, \Omega_h^{00} \\
  u_{h }^{k+1} & = r_h(u_{\tilde{m}}) \, \text{on} \, \partial \Omega_h^0,
\end{split}
\end{align}
and
\begin{align} 
\begin{split}  \label{time-marching3X}
-  \tir{\nu} \Delta_h u_{h }^{k+1} & = -  \tir{\nu} \Delta_h u_{h }^k + \frac{1}{d} \div_h [ (\cof \sym \mathcal{H}_d u_{h}^k ) D_h u_{h}^k ]  - r_h(f_{\tilde{m}}) \, \text{in} \, \Omega_h^{00} \\
  u_{h }^{k+1} & = r_h(u_{\tilde{m}}) \, \text{on} \, \partial \Omega_h^0.
\end{split}
\end{align}

We recall that the parameter $\tilde{m}$ was defined in section \ref{methodns}. 
Intuitively Problems \eqref{m1h-altX} and \eqref{m1hX} discretize the Monge-Amp\`ere equation in the interior of the domain where the non smooth solution can be approximated by smooth functions which solve related Monge-Amp\`ere equations. It is clear that since \eqref{m1h-altX} and \eqref{m1hX} are very close to \eqref{m1h-alt} and \eqref{m1h}, and with the choice of the small parameter $\delta$ introduced in section \ref{methodns}, numerical experiments with the latter would indicate convergence for non smooth solutions. 

The following lemma is essential to our methodology

\begin{lemma} \label{essential}
A sequence of (discrete) convex functions which is locally uniformly bounded has a subsequence which converges uniformly on compact subsets to a (discrete) convex function.

\end{lemma}

\begin{proof}
We consider separately the cases of a sequence $u_m$ of convex functions, a sequence $(u_{m h})_m$ of discrete convex functions and a sequence $u_{h_l}$ of discrete convex functions. 

A sequence $u_m$ of convex functions is locally equicontinuous by  \cite[Lemma 3.2.1]{Guti'errez2001}, c.f. \cite{Awanou-Std04} for details. If the sequence is also locally uniformly bounded, the result follows from the Arzela-Ascoli theorem \cite[p. 179]{Royden}.

If we consider a sequence $(u_{m h})_m$ of discrete convex functions, for fixed $h$ the number of grid points is finite and the result follows from the Bolzano-Weierstrass theorem.

If the sequence $u_{h_l}$ is a sequence of discrete convex mesh functions in the sense that $\tir{\mathcal{H}}_d(u_{h_l})(x)$ is a positive matrix for all $x \in \Omega_h^0$, the result is given by \cite[Corollary 4.8]{Aguilera2008}
and the Arzela-Ascoli theorem (which requires only local uniform boundedness). Since  $\mathcal{H}_d(u_{h_l})(x)$ and  $\tir{\mathcal{H}}_d(u_{h_l})(x)$ have the same diagonal elements, the discrete analogue of local equicontinuity \cite[(2.2) and p. 22]{Aguilera2008} also holds when one requires that $\mathcal{H}_d(u_{h_l})(x)$ is a positive matrix for all $x \in \Omega_h^0$, that is the result also holds in that case.


\end{proof}



We make the usual abuse of notation of denoting by $C$ a generic constant which does not depend on $h$. 

\section{General framework for convergence of standard discretizations to the Aleksandrov solution}  \label{gen-frame}

\subsection{The Aleksandrov solution} \label{Aleks-def}
Let $K(\Omega)$ denote the cone of convex functions on $\Omega$ and let us denote by $B(\Omega)$ the space of Borel measures on $\Omega$. We define the mapping
\begin{align*}
M: C^2(\Omega) \cap K(\Omega) \to B(\Omega), M[v](B) = \int_B \det D^2 v(x) \ud x,
\end{align*}
where $B$ is a Borel set.

The topology on $K(\Omega)$ is the topology of compact convergence, i.e. for $v_m, v \in K(\Omega)$, $v_m$ converges to $v$ if and only if $v_m$ converges to $v$ uniformly on compact subsets of $\Omega$. The topology on 
$B(\Omega)$ is induced by the weak convergence of measures.

\begin{defn} A sequence $\mu_m$ of Borel measures converges weakly to a Borel measure $\mu$ if and only if
$$
\int_{\Omega} p(x)\ud \mu_m  \to \int_{\Omega} p(x) \ud \mu,
$$
for every continuous function $p$ with compact support in $\Omega$.
\end{defn}
If the measures $\mu_m$ have density $a_m$, and $\mu$ has density $a$, we have
\begin{defn}
Let $a_m, a \geq 0$. The sequence $a_m$ converges weakly to  $a$ as measures if and only if 
$$\int_{\Omega }a_m p \ud x \to \int_{\Omega } a p \ud x, $$ 
for all continuous functions $p$ with compact support in $\Omega$. 
\end{defn}

The mapping $M$ extends uniquely to a continuous operator on $K(\Omega)$,  \cite[Proposition 3.1]{Rauch77}. This notion of Monge-Amp\`ere measure can be shown to be equivalent to the one used in \cite{Guti'errez2001,Hartenstine2006}. The proof is given by  \cite[Proposition 3.4]{Rauch77}. 
We have

\begin{lemma}[Lemma 1.2.3 \cite{Guti'errez2001}] \label{weak-s} Let $v_m$ be a sequence of convex functions in $\Omega$ such that $v_m \to v$ uniformly on compact subsets of $\Omega$. Then the associated Monge-Amp\`ere measures $M [v_m]$ tend to $M[v]$ weakly. 
\end{lemma}

\begin{defn}
A convex function $u \in C(\tir{\Omega})$ is said to be an Aleksandrov solution of \eqref{m1} if $u=g$ on $\partial \Omega$ and $M[u]$ has density $f$.
\end{defn}

We have

\begin{thm} [Theorem 1.1 \cite{Hartenstine2006} ] \label{weak-cont}
Let $\Omega$ be a bounded convex domain of $\R^d$ and assume that $g$ can be extended to a function $\tilde{g} \in C(\tir{\Omega})$ which is  convex in $\Omega$. Then if $f \in L^1(\Omega)$ , \eqref{m1} has a unique convex Aleksandrov solution in $C(\tir{\Omega})$ which assumes the boundary condition in the classical sense.
\end{thm}


\subsection{Convergence of the discretization} \label{cvg}
Let $\Omega_s$ denote a sequence of smooth uniformly convex domains increasing to $\Omega$, i.e. $\Omega_s \subset \Omega_{s+1} \subset\Omega$ and $d(\partial \Omega_s, \partial \Omega) \to 0$ as $s \to \infty$. Here $d(\partial \Omega_s, \partial \Omega)$ denotes the distance between $\partial \Omega_s$ and $\partial \Omega$. For the special case $\Omega=(0,1)^2$, a construction was done in \cite{Sulman11}. A general construction follows from the approach in \cite{Blocki97}.

We recall that $f_m$ and $g_m$ are $C^{\infty}(\tir{\Omega})$ functions such that  $0 < c_2 \leq f_m \leq c_3, f_m \to f$ and $g_m \to \tilde{g}$ uniformly on $\tir{\Omega}$. 
The sequences $f_m$ and $g_m$ can be constructed by a standard mollification.

Recall from section \ref{methodns} that we choose $\tilde{m}$ such that $|u(x) - u_{\tilde{m}}(x)| < \delta$ for all $x \in \Omega$, where $\delta $ is a small parameter. And we are interested in convergence of the discretization to the solution $u_{\tilde{m}}$ of \eqref{m1m}.

By  \cite{Caffarelli1984}, 
the problem \eqref{m11}
has a unique convex solution $u_{\tilde{m} s} \in C^{\infty}(\tir{\Omega}_s)$. 
As $s \to \infty$, the sequence $u_{\tilde{m} s}$  converges uniformly  on compact subsets of $\widetilde{\Omega}$ to the unique convex solution $u_{\tilde{m}} \in C(\tir{\widetilde{\Omega}})$ of the problem \eqref{m1m}  \cite{Awanou-Std04}. 

We have by the interior Schauder estimates,  \cite[Theorem 4]{Dinew} and \cite{Awanou-Std04} for details,
\begin{equation} \label{int-Schauder}
||u_{\tilde{m} s}||_{C^2(K)} \leq C_{\tilde{m}},
\end{equation}
where the constant $C_{\tilde{m}}$ depends on $\tilde{m}, c_2$, $\widetilde{\Omega}$, $d(K, \partial \Omega)$, $f_{\tilde{m}}$ and $\max_{x \in \Omega} |u_{\tilde{m} s}(x)|$. Moreover, by a bootstrapping argument 
we have
\begin{equation} \label{int-Schauder2}
||u_{\tilde{m} s}||_{C^5(K)} \leq C_{\tilde{m}},
\end{equation}
as well.

Let us use the notation $M_h[v_h]$ for a discrete Monge-Amp\`ere operator applied to the grid function $v_h$. We consider the following analogue of \eqref{m1h-alt} and \eqref{m1h} 
\begin{align} \label{m1hstd}
M_h[u_h]=r_h(f_{\tilde{m}}) \, \text{in} \, \Omega_h^{00}, u_{h} =r_h(u_{\tilde{m}}) \, \text{on} \, \partial \Omega_h^0.
\end{align}
We can now prove the main result of this paper

\begin{thm} \label{l2-thm2}
The problem \eqref{m1hstd} has a  unique local  discrete convex solution $u_{h}$ which 
converges uniformly on compact subsets of $\widetilde{\Omega}$ 
to the unique convex solution $\tilde{u}$ of \eqref{m1m} 
as $h \to 0$.
\end{thm}

\begin{proof}
Without loss of generality, we assume that the discrete Hessian takes the form $\mathcal{H}_d(v_h)$.

Recall that
$$
\Omega^{00}_h  \subset  \widetilde{\Omega} \cap \mathbb{Z}_h \subset \Omega_s \ \text{and} \   \partial  \Omega^{0}_h \subset \widetilde{\Omega} \cap \mathbb{Z}_h \subset \Omega_s.
$$

{\bf Part 1}: Existence of a discrete convex solution $u_h$

By Theorem \ref{first-cvg} and Proposition \ref{second-cvg}, applied to the problem \eqref{m11}, there exists a unique local solution $u_{\tilde{m} s,h}$ to the problem
\begin{align} \label{m11h}
M_h[ u_{\tilde{m} s,h} ]  =r_h(f_{\tilde{m}}) \, \text{in} \, \Omega_h^{00}, u_{\tilde{m} s,h} =r_h(u_{\tilde{m}}) \, \text{on} \, \partial \Omega_h^0.
\end{align}
For fixed $h$, the number of grid points is finite. Thus by Lemma \ref{essential}, there exist a subsequence  $s_q$ such that $u_{\tilde{m} s_q,h}$ converges pointwise (and hence uniformly on compact subsets of $\Omega^{00}_h$) to a mesh function $u_h$. 

By construction $\partial  \Omega^{0}_h \subset \Omega_s$ and hence for  $x \in \partial \Omega_h^0$, $u_{h}(x)=r_h(u_{\tilde{m}})(x)$.
By taking pointwise limits in \eqref{m11h}, we get that $u_h$ solves \eqref{m1hstd}. 

Since $f_m \geq c_2 >0$, as a consequence of Lemmas \ref{cone} and \ref{cone2}, $\lambda_1(D^2 u_{\tilde{m} s,h}(x)) \geq c_4 >0$ for all $x \in \Omega_h^{00}$ for a constant $c_4$ independent of $h$ and for $s$ sufficiently large. 
But $\lambda_1(\mathcal{H}_d u_{\tilde{m} s,h}(x))$ is the solution of a polynomial equation with coefficients which are combinations of entries of $(\mathcal{H}_d u_{\tilde{m} s,h}(x))_{i,j=1,\ldots,d}$. By continuity of the roots of a polynomial as a function of its coefficients \cite{Harris87}, taking a limit as $s_q \to \infty$, we obtain that $\lambda_1(\mathcal{H}_d u_{h}(x)) \geq 0$ for all $x \in \Omega_h^{00}$. That is, $u_{h}$ is also discrete 
convex. Since $r_h(f) \geq c_0 >0$, $u_{h}$ is discrete strictly convex.

For the local uniqueness of the discrete solution $u_{h}$, we note that the fixed point argument of section \ref{time1} can be repeated in the ball $B_{\rho}(u_{h})$ since $u_{h}$ is a discrete strictly convex function, c.f. Lemma \ref{last-lem}. 
By a similar argument, local uniqueness holds if one uses the discrete Hessian $\tir{\mathcal{H}_d}$ discussed in \cite{Awanou-k-Hessian12}. We conclude that $u_{\tilde{m} s,h}$ converges uniformly on compact subsets of $\Omega^{00}_h$ to $u_h$ as $s \to \infty$.

{\bf Part 2}: Uniform convergence on compact subsets of $\Omega$ of a subsequence $u_{h_l}$ to a convex function  $v$  $\in C(\widetilde{\Omega})$.

This is a direct consequence of the error estimates of Theorem \ref{first-cvg} and Proposition \ref{second-cvg}, the interior Schauder estimate $||u_{\tilde{m} s}||_{C^5(\widetilde{\Omega})} \leq C_{\tilde{m}}$ and
Lemma \ref{essential}. 
The continuity of $v$ on $\widetilde{\Omega}$ follows from its convexity, Theorem \ref{first-cvg},  Proposition \ref{second-cvg} and \eqref{int-Schauder2} which imply that $u_{\tilde{m} s,h}$ and $u_h$, hence $v$ are locally finite.

{\bf Part 3}: The continuous 
convex function $v$ is equal to the Aleksandrov solution $\tilde{u}$ of \eqref{m1m}.

Let $K$ be a compact subset of $\widetilde{\Omega}$ and let $\epsilon >0$. Since $u_{h_l}$ converges uniformly on $K$ to $v$, $\exists l_0$ such that $\forall l \geq l_0$ $|u_{h_l}(x) - v(x)|< \epsilon/6$ for all $x \in K \cap \Omega_h^{00}$.
 
By definition $u_{h_l}$ is the uniform limit on $K \cap \Omega_h^{00}$ of $u_{\tilde{m} s,h_l}$ as $s \to \infty$.  Thus $\exists s_l$ such that $\forall s \geq s_l$ $|u_{\tilde{m} s,h_l}(x) - u_{h_l}(x)|< \epsilon/6$ for all $x \in K \cap \Omega_h^{00}$.

By Theorem \ref{first-cvg},  Proposition \ref{second-cvg} and \eqref{int-Schauder2} we have on $K$ 
$|u_{\tilde{m} s,h_l}(x) - u_{\tilde{m} s}(x)|\leq C h_l^2$ for all $x \in K \cap \Omega_h^{00}$. We recall that the constant $C$ is independent of $s$ but depends on $\tilde{m}$ and $\widetilde{\Omega}$.

By the uniform convergence of $u_{\tilde{m} s}$ to $u_{\tilde{m}}$, we may assume that $|u_{\tilde{m}}(x) - u_{\tilde{m} s}(x)| < \epsilon/6$ for all $x \in K$.

We conclude that for $\forall l \geq l_0$, $\exists s_l$ such that $\forall  s \geq s_l$ $|u_{\tilde{m}}(x) - v(x) | <  \epsilon/2 + C h_l^2$ for all $x \in K \cap \Omega_h^{00}$.

For $x \in K$, if necessary by choosing a sequence $x_{h_l}$ such that $x_{h_l} \to x$ as $l \to \infty$, we get for all $\epsilon >0 $ $|u_{\tilde{m}}(x) - v(x) | < \epsilon$. We conclude that $\tilde{u}=u_{\tilde{m}}=v$ on $K$.
We have by construction $\tilde{u}=v$ on $\partial \widetilde{\Omega}$. This proves that $\tilde{u}=v$.

{\bf Part 4}: Finishing up.

By the unicity of the solution $\tilde{u}$ of \eqref{m1m} we conclude that $u_h$ converges uniformly on compact subsets of $\widetilde{\Omega}$ to $\tilde{u}$.

\end{proof}

It follows from Lemma \ref{last-lem} that the solution $u_{h}$ of \eqref{m1h} can also be computed by the time marching method \eqref{time-marching3}.

\section{Additional notation and preliminaries } \label{notation+}

\subsection{Grid functions and differential operators}
For $v \in C^2(\tir{\Omega})$, by a Taylor series expansion, we have 
\begin{align} \label{first-order}
 \frac{\partial v}{\partial x_i} =  \partial^i_{-}  r_h(v) + O(h), \,  \frac{\partial v}{\partial x_i} =  \partial^i_{+}  r_h(v) + O(h), i=1,\ldots,d. 
\end{align}
and for $v \in C^4(\tir{\Omega})$
\begin{align}  \label{second-order}
\begin{split}
 \frac{\partial^2 v}{\partial x_i^2 } =  \partial^i_{-} \partial^i_{+} r_h(v) + O(h^2), 
 \frac{\partial^2 v}{\partial x_i \partial x_j} =  \partial^j_{-} \partial^i_{+} r_h(v) + O(h),  i,j=1,\ldots,d, i \neq j. 
 \end{split}
\end{align}

We now discuss key properties of the continuous analogues of the operators $\div_h, D_h$ and $\tir{D}_h $ which need to be modeled at the discrete level. 

For a vector field $v=(v_i)$, we define $Dv$ as the matrix field with $(D v)_{ij} = \partial v_i/\partial x_j$. Given a matrix field $A$, we define $\div A$ as the vector field resulting from the application of the operator $\div$ to each row, i.e. $(\div A)_i = \sum_{j=1}^d \partial A_{i j} / \partial x_j,  i=1,\ldots,d$. The Frobenius inner product of two matrices $A=(A_{ij})$ and $B=(B_{ij})$ is defined as  $A: B=\sum_{i,j=1}^n A_{ij} B_{ij}$. 
We recall that the cofactor matrix $\cof A$ of the matrix $A$ is defined by $(\cof A)_{ij}=(-1)^{i+j} \det(A)_i^j$ where $\det(A)_i^j$ is the determinant of the matrix obtained from $A$ by deleting the $i$th row and the $j$th column.

For a $d \times d$ matrix $A$, using the row expansion definition of determinant, one obtains
\begin{equation} \label{obs1}
d \det A =(\cof A):A, 
\end{equation}
and for a vector field $v$ and matrix field $A$, one obtains using the product rule of differentiation
\begin{equation}  \label{obs2}
\div (A v) = (\div A^T) \cdot v + A: (D v)^T.
\end{equation}
 
For $v \in C^3(\Omega)$, we have the divergence-free row property of the cofactor matrix, \cite[p. 440 ]{Evans1998} 
\begin{equation}  \label{obs3}
\div \cof D^2 v =0.
\end{equation}
It follows from \eqref{obs1} that for a $C^2(\Omega)$ function $v$,
\begin{equation} \label{obs11}
 \det D^2 v = \frac{1}{d}(\cof D^2 v):D^2 v. 
\end{equation}
Using \eqref{obs2}, \eqref{obs3} and the symmetry of $\cof D^2 v$ and $D^2 v$ one obtains
for  a $ C^3(\Omega)$ function $v$
\begin{equation} \label{obs22}
 \det D^2 v = \frac{1}{d} \div [(\cof D^2 v) D v]. 
\end{equation}
We have the following lemma which says that the compatible discretization is first order consistent. The proof requires only elementary computations and is given in section \ref{appendix}.

\begin{lemma} \label{det-ap-lem}
We have
\begin{align} 
 \frac{1}{d}  \div_h [(\cof \mathcal{H}_d (r_h v) )^T D_h r_h v](x) - \det D^2 v (x) = O(h) \label{ap03} \\
 \frac{1}{d}  \div_h [(\cof \sym \mathcal{H}_d (r_h v) ) D_h r_h v](x) - \det D^2 v (x) = O(h). \label{ap04}
\end{align}

\end{lemma}

\subsection{Discrete norms}

Analogues of the Sobolev spaces can be defined on $\Omega_h$. We start with the analogue of the $L^2$ 
inner product and norm. For $v_h, w_h \in \mathcal{M}(\Omega_h)$ we define
\begin{align*}
\<v_h,w_h \> = h^d \sum_{x \in \Omega_h^0} v_h(x) w_h(x) \, \text{and} \, ||v_h||_{0,h} = \sqrt{\<v_h,v_h\>}.
\end{align*}
Analogously, put
\begin{align*}
||v_h||_{1,h}  = \bigg( ||v_h||_{0,h}^2 + \sum_{i=1}^d ||  \partial^i_{+}  v_h||^2_{0,h}\bigg)^{\frac{1}{2}}, \
|v_h|_{1,h}  = \bigg( \sum_{i=1}^d ||  \partial^i_{+}  v_h||^2_{0,h}\bigg)^{\frac{1}{2}}.
\end{align*}
We define
\begin{align*}
L^2(\Omega_h)  & = \{ \, v_h \in \mathcal{M}(\Omega_h),  ||v_h||_{0,h} < \infty  \, \}, \
H^1(\Omega_h)  = \{ \, v_h \in \mathcal{M}(\Omega_h),  ||v_h||_{1,h} < \infty  \, \} \ \text{and} \ \\
H_0^1(\Omega_h)  & = \{ \, v_h \in H^1(\Omega_h),  v_h = 0 \, \text{on} \,  \partial \Omega_h \, \}.
\end{align*}

We define on $ \mathcal{M}(\Omega_h)$ the  semi-norms
\begin{align*}
|v_h|_{2,\infty,h} & = \max \{ \,\partial^j_- \partial^i_+ v_h(x), x \in \Omega_h^0, i,j=1,\ldots, d \, \} \\
|v_h|_{1,\infty,h} & = \max \{ \, \partial^i_+ v_h(x), x \in \Omega_h^0, i=1,\ldots, d \, \}.
\end{align*}
We will need the following related norm which takes into account the second order discrete derivatives
\begin{align} \label{max-2nd}
||v_h||_{2,\infty,h} & = \max \{ \,v_h(x), \partial^i_+ v_h(x), \partial^j_+ \partial^i_+ v_h(x), x \in \Omega_h^0, i,j=1,\ldots, d \, \}, 
\end{align}

We will also need the maximum norm
\begin{align} \label{max-1st}
|v_h|_{0,\infty,h}  = \max \{ \, v_h(x), x \in \Omega_h^0 \, \}.
\end{align}

Using the definitions, it is not difficult to check that
\begin{align} \label{int-part}
\< \partial^i_+ v_h, w_h \> =  -\<  v_h, \partial^i_- w_h \> , i=1,\ldots,d, v_h, w_h \in H_0^1(\Omega_h).
\end{align}

Since $|v_h|_{0,\infty,h}^2 \leq \sum_{x \in \Omega_h^0 } |v_h(x)|^2$, we obtain
\begin{equation} \label{inverse0-inf-0}
|v_h|_{0,\infty,h} \leq h^{-\frac{d}{2}} ||v_h||_{0,h}.
\end{equation}
Using \eqref{first-order}, \eqref{second-order} and \eqref{inverse0-inf-0}, we get
\begin{align}
|v_h|_{2,\infty,h} & \leq C h^{-1} |v_h|_{1,\infty,h} \leq C h^{-\frac{d}{2}-1} |v_h|_{1,h}, \label{inverse1}
\end{align}
and
\begin{align*}
|v_h|_{1,\infty,h} & \leq  C h^{-\frac{d}{2}} |v_h|_{1,h}. 
\end{align*} 
Inequality \eqref{inverse1} is an inverse type estimate with constant given by
\begin{align} \label{inverse-C}
C_{inv}(h) = C h^{-\frac{d}{2}-1}.
\end{align}
We have the discrete Poincare's inequality, see for example \cite[Lemma 3.1]{Chung97} 
\begin{lemma} \label{poincare}
There exists a constant $C_p >0$ independent of $h$ such that for $v_h \in H_0^1(\Omega_h)$, 
$$
|v_h|_{1,h}  \geq C_p ||v_h||_{0,h}. 
$$
\end{lemma}
By the integration by parts formula \eqref{int-part}, we obtain for $v_h \in H_0^1(\Omega_h)$, 
\begin{align} \label{weight}
-\<\div_h (D_h v_h), v_h \> = |v_h|_{1,h}^2.
\end{align}
And for $v_h, w_h \in \mathcal{M}(\Omega_h)$,
\begin{align} \label{schwarz}
|\<v_h, w_h \> | & \leq ||v_h||_{0,h} ||w_h||_{0,h} \leq C |v_h|_{0,\infty,h} ||w_h||_{0,h}. 
\end{align}

\subsection{On the cone of discrete convex functions}
Let us denote by $K(\Omega_h)$ the cone of discrete convex functions on $\Omega_h$ and by $K_0(\Omega_h)$ the cone of  discrete strictly convex functions on $\Omega_h$.

 By the continuity of the eigenvalues of $A$ as a function of its entries, 
\cite[Theorem 1 and Remark 2 p. 39]{Hoffman53}, we have for two symmetric $d \times d$ matrices $A$ and $B$, 
\begin{equation} \label{cont-eig}
|\lambda_k(A) - \lambda_k(B)| \leq d \max_{i,j=1,\ldots,d} |A_{ij} - B_{ij}|, k=1, d.
\end{equation}
Since for two $d \times d$ matrices $A$ and $B$
\begin{align*}
\max_{i,j=1,\ldots,d} |(\sym A)_{ij} - (\sym B)_{ij} | = \frac{1}{2} \max_{i,j=1,\ldots,d} |A_{ij} - B_{ij} + A^T_{ij} - B^T_{ij}  | \leq  \max_{i,j=1,\ldots,d} |A_{ij} - B_{ij}|,
\end{align*}
it follows that for $v_h, w_h \in \mathcal{M}(\Omega_h)$, 
\begin{align} \label{continuity-eig}
|\lambda_k(\sym \mathcal{H}_d ( v_h ) ) - \lambda_k( \sym \mathcal{H}_d ( w_h ) ) | & \leq d |v_h-w_h|_{2,\infty,h}, k=1,d.
\end{align}

\begin{lemma} \label{cone}
Let $v_h \in K_0(\Omega_h)$ and assume that
$$
\lambda_1(\sym \mathcal{H}_d ( v_h ) ) \geq C_{v_h} >0 \, \text{on} \, \Omega_h^0.
$$
Then
$$
\bigg\{ \, w_h \in \mathcal{M}(\Omega_h), |v_h-w_h|_{1,h} \leq \frac{C_{v_h} }{2 d C_{inv}(h)}  \, \bigg\} \subset K_0(\Omega_h).
$$
\end{lemma}
\begin{proof}
We have
\begin{align*}
|v_h-w_h|_{2,\infty,h} \leq  C_{inv}(h) |v_h-w_h|_{1,h}.
\end{align*}
Thus if $|v_h-w_h|_{1,h} \leq C_{v_h}/(2 d C_{inv}(h))$, we have
$$
|\lambda_1(\sym \mathcal{H}_d ( v_h ) ) - \lambda_1( \sym \mathcal{H}_d ( w_h ) ) | \leq \frac{C_{v_h}}{2},
$$
and so
\begin{align*}
 \lambda_1( \sym \mathcal{H}_d ( w_h ) ) \geq  \lambda_1( \sym \mathcal{H}_d ( v_h ) ) - \frac{C_{v_h}}{2} \geq \frac{C_{v_h}}{2} >0.
\end{align*}
This proves the result.
\end{proof}

The next lemma says that if $v \in C^2(\tir{\Omega})$ is strictly convex, a lower bound on the smallest eigenvalue of $\sym  \mathcal{H}_d  r_h v $ is independent of $h$, for $h$ sufficiently small.

\begin{lemma} \label{cone2}
Let  $v \in C^2(\tir{\Omega})$ be a strictly convex function. Assume that 
$$
r \leq \lambda_1(D^2 v) \leq \lambda_d(D^2 v) \leq R,
$$ 
on $\Omega$ for constants $r, R >0$. Then for $h$ sufficiently small
$$
\frac{r}{2} \leq \lambda_1(\sym \mathcal{H}_d  r_h v) \leq  \lambda_d(\sym \mathcal{H}_d  r_h v) \leq \frac{3 R}{2}.
$$
\end{lemma}

\begin{proof}
The proof is similar to the one of Lemma \ref{cone} using \eqref{continuity-eig}. It is enough to prove that
$$
\max_{i,j=1,\ldots,d} \bigg|\frac{\partial^2 v}{\partial x_i \partial x_j} - (\mathcal{H}_d  r_h v)_{ij}  \bigg| \leq \max \bigg( \frac{r}{2 d}, \frac{R}{ d}  \bigg).
$$
But this holds for $h$ sufficiently small using a Taylor series expansion.
\end{proof}

\begin{lemma} \label{lem-1}
Let $A$ be a symmetric matrix such that 
$$
0< r \leq \lambda_1(A)  \leq \lambda_d(A) \leq R.
$$
Then
$$
r' \leq \lambda_1(\cof A) \leq \lambda_d (\cof A)  \leq R', 
$$ 
with $r'=(r)^{d}/R$ and $R'=(R)^{d}/r$.
\end{lemma}

\begin{proof} 
Since $A$ is an invertible  matrix,
$\cof A = (\det A) (A^{-1})^T$. Recall that $A$ and $A^T$ have the same set of eigenvalues. Hence  the eigenvalues of $\cof A$ are of the form $ \det A/\lambda_i$ where 
$\lambda_i, i=1,\ldots,d$ is an eigenvalue of $A$. Since $r^d \leq \det A \leq R^d$, we get the result.

\end{proof}

\section{Convergence of the time marching method for the compatible discretization} \label{time1}

In this section, we assume that \eqref{m1} has a strictly convex solution $u \in C^4(\tir{\Omega})$. We prove that the problem
 \begin{align} \label{m0h}
\frac{1}{d} \div_h [ (\cof \sym \mathcal{H}_d u_{h} ) D_h u_{h} ]  =r_h(f) \, \text{in} \, \Omega_h^0, u_{h} =r_h(\tilde{g}) \, \text{on} \, \partial \Omega_h,
\end{align}
has a  discrete strictly convex solution in
$$
B_{\rho}(r_h u) = \{ \, v_h \in \mathcal{M}(\Omega_h), |v_h-r_h(u)|_{1,h} \leq \rho \, \},
$$
for $\rho=O(h^{1+d/2} )$ and $h$ sufficiently small.



As with \cite{Awanou-Std01} we use a ''rescaling argument''. Let $\alpha >0$ be a positive parameter. Note that if $v_h \in K_0(\Omega_h)$, we also have $\alpha v_h \in K_0(\Omega_h)$.

Since $u \in C^2(\tir{\Omega})$ is a strictly convex function, there exists positive constants $r$ and $R$ such that
$$
r \leq \lambda_1(D^2 u) \leq \lambda_d(D^2 u) \leq R,
$$ 
on $\Omega$. By Lemma \ref{cone2}, for $h$ sufficiently small, on $\Omega_h^0$
$$
\frac{r}{2} \leq \lambda_1( \sym \mathcal{H}_d  r_h u) \leq  \lambda_d( \sym \mathcal{H}_d  r_h u) \leq \frac{3 R}{2}.
$$
By Lemma \ref{cone}, and using the estimate \eqref{inverse-C}, there exists $C_0 >0$ such that for
\begin{equation} \label{const-c0}
|v_h-r_h(u)|_{1,h} \leq C_0 h^{\frac{d}{2} + 1},
\end{equation}
$v_h \in  K_0(\Omega_h)$. Moreover as in the proof of Lemma \ref{cone}, one shows that there exists $r_1 >0$ independent of $h$ such that $\lambda_1(\sym \mathcal{H}_d  v_h) \geq r_1$. Similarly, one can show that
$\lambda_d(\sym \mathcal{H}_d  v_h) \leq R_1$ for some constant $R_1 >0$ independent of $h$. In summary
\begin{equation*} 
\text{for} \ |v_h-r_h(u)|_{1,h} \leq C h^{\frac{d}{2} + 1}, r_1 \leq \lambda_1(\sym \mathcal{H}_d  v_h) \leq \lambda_d(\sym \mathcal{H}_d  v_h) \leq R_1.
\end{equation*}
In addition, by Lemma \ref{lem-1}, there exists positive constants $r', R'$ independent of $h$ such that
\begin{align} \label{contrac-p0}
r' \leq \lambda_1(\cof \sym \mathcal{H}_d v_{h}) \leq \lambda_d (\cof \sym \mathcal{H}_d v_{h})  \leq R'.
\end{align}
Put 
$$
\nu = (r'+R')/(2 d).
$$

\begin{rem} \label{C0rem}
The constant $C_0$ is up to a constant a lower bound of $\lambda_1(D^2 u)$. For each compact subset $K$ of $\Omega$ we have by \eqref{cont-eig} with $B=0$ 
$$
\inf_{x \in K} |\lambda_1(D^2 u)(x)| \leq C ||u||_{C^2(K)}.
$$

\end{rem}

 We define a mapping $T_h: \mathcal{M}(\Omega_h) \to \mathcal{M}(\Omega_h)$ characterized by
\begin{align*}
-\nu \Delta_h T_h(\alpha v_h) & = -\nu \Delta_h \alpha v_h +\alpha^d\bigg( \frac{1}{d} \div_h [ (\cof \sym \mathcal{H}_d v_{h} ) D_h v_{h} ]  - r_h(f) \bigg) \\
T_h(\alpha v_h) & = \alpha r_h(\tilde{g})  \, \text{on} \, \partial \Omega_h.
\end{align*}
Since $T_h(\alpha v_h)$ solves a discrete Poisson equation, it is well defined.
\begin{rem} \label{fixed-rem}
If $\alpha u_h$ is a fixed point  of $T_h$, then $u_h$ solves \eqref{m0h}.
\end{rem}

Next, we estimate the amount by which the mapping $T_h$ moves the center $\alpha r_h(u)$ of $\alpha B_{\rho}(r_h u)$.
\begin{lemma} \label{mov-ball}
We have
$$
|T_h(\alpha r_h(u) )- \alpha r_h(u) |_{1,h} \leq \frac{C_1}{\nu} \alpha^d h.
$$
\end{lemma}

\begin{proof}
We have
\begin{align} \label{mov-ball-p1}
- \Delta_h(T_h(\alpha r_h u) - r_h(u) ) & = \frac{\alpha^d}{ \nu} \bigg(  \frac{1}{d} \div_h [ (\cof \sym \mathcal{H}_d v_{h} ) D_h v_{h} ]
- r_h (\det D^2 u)
\bigg). 
\end{align}
Let $z_h = T_h(\alpha r_h u) - r_h(u) \in H_0^1(\Omega_h)$. Taking the inner product of \eqref{mov-ball-p1} with $z_h$ and using \eqref{weight} and \eqref{schwarz} we obtain
\begin{align*}
|z_h|_{1,h}^2 \leq \frac{C}{\nu} \alpha^d \bigg|\frac{1}{d} \div_h [ (\cof \sym \mathcal{H}_d v_{h} ) D_h v_{h} ]
- r_h (\det D^2 u) \bigg|_{0,\infty,h} |z_h|_{1,h}.
\end{align*}
By  \eqref{ap03}, we get
$$
 |z_h|_{1,h} \leq \frac{C}{\nu} \alpha^d h.
$$
\end{proof}

We now give a contraction property for $T_h$.
\begin{lemma} \label{con-lem}
For $h$ sufficiently small,  $\alpha=h^{(3+d/2)/(d-1)}$ and $\rho \leq C_0 h^{d/2+1}$, $T_h$ 
is a strict contraction mapping in the ball
$\alpha B_h(\rho)$, 
i.e. for $v_h, w_h \in  B_h(\rho)$
$$
|T_h(\alpha v_h)-T_h(\alpha w_h)|_{1,h} \leq a |\alpha v_h-\alpha w_h|_{1,h}, 0<a< 1.
$$
The constant $a$ takes the form $\beta+ C h (\rho +C)^{d-1}$ for $0<\beta < 1$.
\end{lemma}
\begin{proof}
For a matrix field $A_h$, we define
$$
|A_h|_{0,\infty,h} = \max_{i,j=1,\ldots,d} |(A_h)_{ij}|_{0,\infty,h}.
$$
Let us denote by $ \cof' $ the Fr\'echet derivative of the mapping $A \to  \cof A $. Since $\cof'(A) (B)$ is a sum of terms each of which is a product of $d-2$ entries from $A$ and is linear in $B$, we have
for $ t \in \R$
\begin{align} \label{meanv-cof}
\begin{split}
|\cof'(t \sym \mathcal{H}_d v_{h} +& (1-t)  \sym \mathcal{H}_d w_{h} )  ( \sym \mathcal{H}_d w_{h} -  \sym \mathcal{H}_d w_{h} )  |_{0,\infty,h} \\
&\qquad \qquad \qquad \leq C |t v_h + (1-t) w_h|_{2,\infty,h}^{d-2} |v_h-w_h|_{2,\infty,h}.
\end{split}
\end{align}
For $v_h=(v_{i,h})_{i=1,\ldots,d}, w_h=(w_{i,h})_{i=1,\ldots,d} \in \mathcal{M}(\Omega_h)^d$, we define
$$
\< v_h, w_h \> = \sum_{i=1}^d \<v_{i,h}, w_{i,h} \>.
$$
We have
\begin{align} \label{contrac-p1}
\begin{split}
-\Delta_h (T_h(\alpha v_h) & -T_h(\alpha w_h))  = -\Delta_h \alpha (v_h - w_h ) \\
& \quad + \frac{\alpha^d}{d \nu} \div_h \bigg\{ 
(\cof \sym \mathcal{H}_d v_{h} ) D_h v_{h} - (\cof \sym \mathcal{H}_d w_{h} ) D_h w_{h}
\bigg\}.
\end{split}
\end{align}
On the other hand
\begin{align} \label{contrac-p2}
\begin{split}
(\cof \sym \mathcal{H}_d v_{h} ) & D_h v_{h} -  (\cof \sym \mathcal{H}_d w_{h} ) D_h w_{h}  = (\cof \sym \mathcal{H}_d v_{h} ) \\ 
& \qquad  (D_h v_{h} - D_h w_{h}) 
+ ( \cof \sym \mathcal{H}_d v_{h} - \cof \sym \mathcal{H}_d w_{h} ) D_h w_h.
\end{split}
\end{align}
Let $z_h = T_h(\alpha v_h)-T_h(\alpha w_h) \in H_0^1(\Omega_h)$. Taking the inner product of \eqref{contrac-p1} with $z_h$ and using \eqref{weight} and \eqref{contrac-p2} we obtain
\begin{align} \label{contrac-p3}
\begin{split}
|z_h|_{1,h}^2 & = \< D_h  \alpha (v_h - w_h ), D_h z_h \>  - \frac{\alpha^d}{d \nu} \< (\cof \sym \mathcal{H}_d v_{h} ) (D_h v_{h} - D_h w_{h}), D_h z_h \> \\
& \qquad - \frac{\alpha^d}{d \nu} \<    ( \cof \sym \mathcal{H}_d v_{h} - \cof \sym \mathcal{H}_d w_{h} ) D_h w_h, D_h z_h \> \\
& = \alpha \<(I - \frac{1}{ d\nu} (\cof \alpha \sym \mathcal{H}_d v_{h} ) ) D_h  (v_h - w_h ), D_h z_h \>  \\
& \qquad \qquad - \frac{\alpha^d}{d \nu} \<    ( \cof \sym \mathcal{H}_d v_{h} - \cof \sym \mathcal{H}_d w_{h} ) D_h w_h, D_h z_h \>,
\end{split}
\end{align}
where $I$ denotes the $d \times d$ identity matrix.

Since $\rho \leq C_0 h^{d/2+1}$, by \eqref{contrac-p0} for $v_h \in B_{\rho}(r_h u)$, we have for $z_h \in \mathcal{M}(\Omega_h)$
\begin{align} \label{ess-step}
r' \alpha^{d-1} |z_h|_{1,h}^2 \leq  \< ( \cof \alpha \sym \mathcal{H}_d v_{h} ) D_h z_h, D_h z_h \> \leq  R' \alpha^{d-1} |z_h|_{1,h}^2. 
\end{align}
Therefore
\begin{align*}
(1- \frac{R' \alpha^{d-1} }{ d\nu}) |z_h|_{1,h}^2 \leq  \< (I - \frac{1}{ d\nu}  (\cof \alpha \sym \mathcal{H}_d v_{h} ) )D_h z_h, D_h z_h \> \leq  (1- \frac{r' \alpha^{d-1}}{ d\nu})  |z_h|_{1,h}^2. 
\end{align*}
We define
$$
\beta = \sup_{z_h \in \mathcal{M}(\Omega_h) \atop |z_h|_{1,h} = 1}  \<(I - \frac{1}{ d\nu}  (\cof \alpha \sym \mathcal{H}_d v_{h} ) )D_h z_h, D_h z_h \>.
$$
Since $\nu =(R'+r')/(2d)$, we have
\begin{align*}
1- \frac{\alpha^{d-1} R'}{ d \nu} & = \frac{r'+R' - 2 R' \alpha^{d-1}}{r'+R'} < 1 \\
1- \frac{\alpha^{d-1} r'}{ d\nu} & =  \frac{r'+R' - 2 r' \alpha^{d-1}}{r'+R'} < 1.
\end{align*}
Thus since for $h$ sufficiently small 
\begin{equation} \label{last-eq1}
\alpha^{d-1} < \frac{r'+R'}{2 R'} \leq \frac{r'+R'}{2 r'},
\end{equation}
we have
$$
0 \leq \beta < 1.
$$

Define $p_h= w_h/|w_h|_{1,h}$ and  $q_h= z_h/|z_h|_{1,h}$ for $w_h \neq 0$ and $v_h \neq 0$. Then 
\begin{align} \label{pp1}
\frac{\<(I - \frac{1}{ d\nu}  (\cof \alpha \sym \mathcal{H}_d v_{h} ) )D_h w_h, D_h z_h \> }{|w_h|_{1,h} |z_h|_{1,h}}  =  \<(I - \frac{1}{ d\nu}  (\cof \alpha \sym \mathcal{H}_d v_{h} ) )D_h p_h, D_h q_h \>.
\end{align}
We can define a bilinear form on $\mathcal{M}(\Omega_h) $ by the formula
\begin{align*}
(p,q) & =\<(I - \frac{1}{ d\nu}  (\cof \alpha \sym \mathcal{H}_d v_{h} ) )D_h p, D_h q \>.
\end{align*}
Then because
$$
(p,q) = \frac{1}{4} ((p+q,p+q) - (p-q,p-q)),
$$
and using the definition of $\beta$, we get
$$
|(p_h,q_h)| \leq \frac{\beta}{4} |p_h+q_h|_{1,h}^2 + \frac{\beta}{4}  |p_h-q_h|_{1,h}^2 = \beta,
$$
since $p_h$ and $q_h$ are unit vectors in the $| . |_{1,h}$ semi-norm. It follows from \eqref{pp1} that for $w_h, z_h \in \mathcal{M}(\Omega_h)$
\begin{equation} \label{pp2}
| \<(I - \frac{1}{ d\nu}  (\cof \alpha \sym \mathcal{H}_d v_{h} )^T )D_h w_h, D_h z_h \> | \leq \beta |w_h|_{1,h} |z_h|_{1,h}.
\end{equation}
We then conclude from \eqref{contrac-p3} that
\begin{align*}
|z_h|_{1,h}^2 \leq \beta |\alpha(v_h -w_h)|_{1,h} |z_h|_{1,h} +  \frac{\alpha^d}{d \nu}   | ( \cof \sym \mathcal{H}_d v_{h} - \cof \sym \mathcal{H}_d w_{h} ) |_{0,\infty,h}  |w_h|_{1,h}  |z_h|_{1,h}.
\end{align*}
By the mean value theorem and \eqref{meanv-cof}, we get
\begin{align*}
|z_h|_{1,h} \leq \beta |\alpha(v_h -w_h)|_{1,h}+ C  \frac{\alpha^d}{d \nu}   (|v_h|_{2,\infty,h}  + |w_h|_{2,\infty,h} )^{d-2}  |v_h-w_h|_{2,\infty,h} |w_h|_{1,h}.
\end{align*}
We note that by triangular inequality 
$$
|w_h|_{1,h} \leq |w_h-r_h(u)|_{1,h} + |r_h(u)|_{1,h} \leq \rho + C ||u||_{C^1(\tir{\Omega})} \leq \rho +C.
$$
Thus using \eqref{inverse1} 
we conclude that
\begin{align*}
|z_h|_{1,h}&  \leq \beta |\alpha(v_h -w_h)|_{1,h}+ \frac{C}{\nu} \alpha^{d} ( \rho +C)^{d-1}  |v_h-w_h|_{2,\infty,h}\\
&  \leq \beta |\alpha(v_h -w_h)|_{1,h}+ \frac{C}{\nu} \alpha^{d} ( \rho +C)^{d-1}  h^{-\frac{d}{2} -1} |v_h-w_h|_{1,h} \\
& =( \beta + \frac{C}{\nu} \alpha^{d-1} ( \rho +C)^{d-1}  h^{-\frac{d}{2} -1} ) |\alpha(v_h -w_h)|_{1,h}.
\end{align*}
Since $\beta <1$, with $\alpha=h^{(3+d/2)/(d-1)}$, we get $C \alpha^{d-1} ( \rho +C)^{d-1}  h^{-d/2 -1}/\nu < 1-\beta$ and $a= \beta + C \alpha^{d-1} ( \rho +C)^{d-1} h^{-d/2 -1}/\nu < 1$ 
for $h$ sufficiently small.

\end{proof}
\begin{lemma} \label{errorest0}
For $h$ sufficiently small, $\rho=  (C_0/2) \, h^{1+d/2} $ and $ \alpha=h^{(3+d/2)/(d-1)}$, 
$T_h$ 
is a strict contraction in $\alpha  B_{\rho}(r_h u)$ and maps $\alpha  B_{\rho}(r_h u)$ into itself.
\end{lemma}

\begin{proof}
We have $\rho \leq C_0  h^{1+d/2}$ for $h \leq 1$. 
Moreover
\begin{align*}
\frac{C_1}{\nu} \alpha^d  h & = \frac{C_1}{\nu} h^{3+\frac{d}{2}} \alpha h = \frac{C_1}{\nu} h^{4+\frac{d}{2}}  \alpha =\frac{2 C_1}{\nu C_0} h^3  \alpha \rho.
\end{align*}
Thus using the expression of $a$ in Lemma \ref{con-lem}, we get 
$$
\frac{2 C_1}{\nu C_0} h^3 + a = \frac{2 C_1}{\nu C_0} h^3 + \beta+ C h (\rho +C)^{d-1} \leq C h + \beta.
$$
Since $\beta <1$, we have for $h$ sufficiently small 
$C_1 \alpha^d  h /\nu \leq (1-a)  \alpha \rho$ . 
Now, let $v_h \in B_{\rho}(r_h u)$. Then by Lemmas \ref{con-lem} and \ref{mov-ball}
\begin{align*}
|T_h(\alpha v_h) - \alpha r_h u |_{1,h} & \leq |T_h(\alpha v_h) - T_h(\alpha r_h u)|_{1,h} +  |T_h(\alpha r_h u)-\alpha r_h u|_{1,h} \\
& \leq a |\alpha v_h - \alpha r_h u |_{1,h} + \frac{ C_1\alpha^d}{\nu}  h\\
& \leq a |\alpha v_h - \alpha r_h u |_{1,h} + (1-a)  \alpha \rho\\
& \leq a \alpha \rho + (1-a)    \alpha \rho \\
& \leq \alpha \rho,
\end{align*}
and we conclude that
$$
|T_h(\alpha v_h) - \alpha r_h u |_{1,h} \leq \alpha \rho.
$$
This proves the result.
\end{proof}

We can now state the main result of this section
\begin{thm} \label{errorest}
Let $u \in C^4(\tir{\Omega})$ be a  strictly convex solution of \eqref{m1}. For $h$ sufficiently small, the discrete Monge-Amp\`ere equation \eqref{m0h}  has a unique discrete strictly convex solution $u_h \in B_{\rho}(r_h u)$ with $$\rho= \frac{C_0}{2} h^{ 1+\frac{d}{2} }, 
$$ 
and
\begin{align} \label{err-estimate}
  |r_h(u)-u_h|_{1,h} &\leq  \frac{C_0}{2} h^{1+\frac{d}{2}},
\end{align} 
Moreover, with a sufficiently close initial guess $u_h^0$, the sequence defined by $u_h^{k+1}=r_h(\tilde{g})$ on $\partial \Omega_h$
\begin{align}  \label{time-marching-0}
\begin{split}
- \frac{\nu}{\alpha^{d-1}} \Delta_d u_h^{k+1} & = - \frac{\nu}{\alpha^{d-1}} \Delta_d u_h^{k} + \frac{1}{d} \div_h [ (\cof \sym \mathcal{H}_d u_{h} ) D_h u_{h}^k ]  -r_h(f) \, \text{in} \, \Omega_h, 
\end{split}
\end{align}
 converges linearly to $u_h$ in the $H^1(\Omega_h)$ norm for $\nu=(R'+r')/(2 d)$ and $\alpha=h^{(3+d/2)/(d-1)}$. 
\end{thm}
\begin{proof}
Since the mapping $T_h$ is a strict contraction which maps $\alpha  B_{\rho}(r_h u)$ into itself, the existence of a fixed point follows from the Banach fixed point theorem. By Remark \ref{fixed-rem}, a fixed point of $T_h$ solves \eqref{m0h}. Moreover the sequence defined by $\alpha u_h^{k+1} = T_h(\alpha u_h^{k})$ converges linearly to $u_h$ for  a sufficiently close initial guess $u_h^0$. This gives \eqref{time-marching-0}.
Finally the convergence rate follows from the expression of $\rho$.
This completes the proof.
 
\end{proof}

\begin{rem} \label{rho23}
Lemma \ref{errorest0} and Theorem \ref{errorest} also hold with $\rho=O(h^{2+d/2} )$ and $\alpha=h^{(3+d/2)/(d-1)}$.
\end{rem}

\begin{rem} \label{quad-cvg}
We have an asymptotic convergence rate in the maximum norm, i.e. using \eqref{inverse0-inf-0} and the discrete Poincare's inequality, we obtain
\begin{align*}
|r_h(u_{}) - u_{h}|_{0,\infty,h} & \le h^{-\frac{d}{2}} ||r_h(u_{}) - u_{h}||_{0,h} \\
& \le C h^{-\frac{d}{2}} |r_h(u_{}) - u_{h}|_{1,h} \leq C h^2, 
\end{align*}
where we used Remark \ref{rho23}. 
Essentially, what is proven in this paper, is that the numerical solution $u_h$ is very close to the interpolant $r_h(u)$. 
\end{rem}

\begin{rem} \label{vip}
The constant $C_{v_h}$ in Lemma \ref{cone} scales linearly with the size of $v_h$. As a consequence, the constant $C_0$ defined in \eqref{const-c0} also scales linearly with the size of $u$. The same thus holds for our error estimate \eqref{err-estimate}. As a practical consequence, if $\delta >0$ is a user's measure of how close an initial guess can be, i.e.  
$|v_h-v_h^0|_{1,h} \leq \delta$ where $v_h$ is the solution of the numerical problem and $v_h^0$ the guess, 
one only needs to solve the rescaled equation
$$
\frac{1}{d} \div_h [ (\cof \sym \mathcal{H}_d(u_{h}) ) D_h \beta u_{h} ] = \beta^d r_h(f)  \, \text{in} \, \Omega_{h}^0, u_{h} =r_h(g) \, \text{on} \, \partial \Omega_{h},
$$
where $\beta$ solves $\beta (C_0/2) h^{1+d/2} = \delta$. Thus essentially, the convergence of the time marching method is independent of the choice of an initial guess. Similar arguments are given in \cite{Awanou-k-Hessian12} for the case when the discretization \eqref{m1h-alt} is used.
 
\end{rem}

\begin{rem}
For the implementation of the compatible discretization introduced in this paper, quadratic extrapolation has to be used to set values which are not defined on the grid.
\end{rem}

We now show that the framework above can be extended to the situation where it is known that the discrete problem has a discrete strictly convex solution.

\begin{lemma} \label{last-lem}
Assume that \eqref{m1h} has a discrete strictly convex solution. Then the solution is unique locally and can be computed by the time marching method \eqref{time-marching3}.
\end{lemma}

\begin{proof}
We indicate how the proofs are adapted.

Recall that $\rho = (C_0/2) h^{1+d/2}$, and now, by Lemma \ref{cone},  $C_0$ takes the form $C_0 = C r_1$ for a constant $C$ independent of $h$ where
$$
r_1 \leq \lambda_1(\mathcal{H}_d u_h(x)) \leq \lambda_1(\mathcal{H}_d u_h(x)) \leq R_1, \forall x \in \Omega^0_h.
$$
Recall that $\nu = (r_1'+R_1')/2$ with $r_1'=r_1^d/R_1$ and $R_1'=R_1^d/r_1$ by Lemma \ref{lem-1}.


We now take
\begin{equation} \label{new-alpha}
\alpha= \frac{ h^{\frac{3(3+\frac{d}{2})}{d-1}} }{\rho + |u_h|_{1,h}  } \nu^{\frac{1}{d-1}}.
\end{equation}
Then
\begin{align*}
\frac{\alpha^{d-1}}{\frac{r_1'+R_1'}{2 r_1'}} & = \frac{h^{3 (3+\frac{d}{2})} r_1'}{(\frac{C_0}{2} h^{1+\frac{d}{2}} + |u_h|_{1,h})^{d-1}}  \leq \frac{h^{3(3+\frac{d}{2})}  \frac{r_1^d}{R_1} }{C r_1^{d-1}  h^{(1+\frac{d}{2})(d-1)} }
= C \frac{r_1}{R_1} h^{10 + \frac{d}{2} (2-d)}.
\end{align*}
Thus $\alpha^{d-1} < (r_1'+R_1')/(2 r_1')$ for $h$ sufficiently small, i.e. \eqref{last-eq1} holds.

For the last step in the proof of Lemma \ref{con-lem}, we have
\begin{align*}
C \frac{\alpha^{d-1}}{\nu} (\rho + |u_h|_{1,h} )^{d-1} h^{-\frac{d}{2}-1} = C h^{8+d},
\end{align*}
and since $\beta <  1$ 
$$
a=C h^{8+d} + \beta <1,
$$
for $h$ sufficiently small. One then obtains that the mapping $T_h$ is a strict contraction in $\alpha B_{\rho}(u_h)$ with $T_h(\alpha  u_h)=\alpha  u_h$. This concludes the proof.

\end{proof}

\section{Numerical results} \label{num}

The computational domain is the unit square $[0,1]^2$. 
The initial guess for the iterations was taken as the finite difference approximation of the solution of $\Delta u = 2 \sqrt{f}$ with boundary condition $u=g$. Numerical errors are in the maximum norm. 

The scheme \eqref{m1h} performs well for the standard tests for strictly convex viscosity solutions of the Monge-Amp\`ere equation. The results are given on Tables \ref{tab1}, \ref{tab2} and Figure \ref{fig1}. 

 For the non smooth solution of Table \ref{tab2}, and $h=1/2^7, \tir{\nu}=850$, when we run the iterative method \eqref{time-marching} for 10000 iterations followed by \eqref{time-marching3} we get an error $5.0530 \ 10^{-4} $, a level of accuracy which had not been achieved before. And it only took  908 seconds or 15 minutes on a 2.5 GHz MacBook Pro.
Although the convergence of the iterative methods is (theoretically upon rescaling the equation) independent of the closeness of an initial guess, for the best results, one should use the second order accurate scheme \eqref{time-marching} to provide an initial guess for the compatible discretization \eqref{time-marching3}. In fact, \eqref{time-marching} works also in the degenerate case $f=0$ and $g(x,y)=|x-1/2|$, Figure \ref{fig2}. For smooth solutions, \eqref{time-marching} appears to be dramatically faster than Newton's method, Table \ref{tab3}.


\begin{table}
\begin{tabular}{c|ccccc} 
 \multicolumn{5}{c}{$h$}\\
$ \tir{\nu}$ & $1/2^2$ &  $1/2^3$ &  $1/2^4$&  $1/2^5$ &  $1/2^6$ \\ \hline
50 & 9.2277   $10^{-3}$ & 6.5555 $10^{-3}$ & 3.9964 $10^{-3}$ & 2.1694 $10^{-3}$ & 5.0688 $10^{-4}$ \\
\end{tabular}
\caption{Smooth solution $u(x,y)=e^{(x^2+y^2)/2}, g(x,y)=e^{(x^2+y^2)/2}$ and $f(x,y) =(1+x^2+y^2)e^{x^2+y^2}$ with the compatible discretization \eqref{time-marching3} }\label{tab1}
\end{table}
\begin{table}
\begin{tabular}{c|cccc} 
 \multicolumn{4}{c}{$h$}\\
$ \tir{\nu}$ & $1/2^3$ &  $1/2^4$ &  $1/2^5$&  $1/2^6$  \\ \hline
150 & 3.9140 $10^{-3}$ & 2.5847 $10^{-3}$ & 1.4879 $10^{-3}$ & 6.3084 $10^{-4}$  \\ 
\end{tabular} 
\caption{Non smooth solution (not in $H^2(\Omega)$) $u(x,y)=-\sqrt{2-x^2-y^2}, g(x,y)=-\sqrt{2-x^2-y^2}$ and $f(x,y) = 2/(2-x^2-y^2)^2$  with the compatible discretization \eqref{time-marching3} } \label{tab2}
\end{table}
\begin{figure}[tbp]
\begin{center}
\includegraphics[angle=0, height=4.5cm]{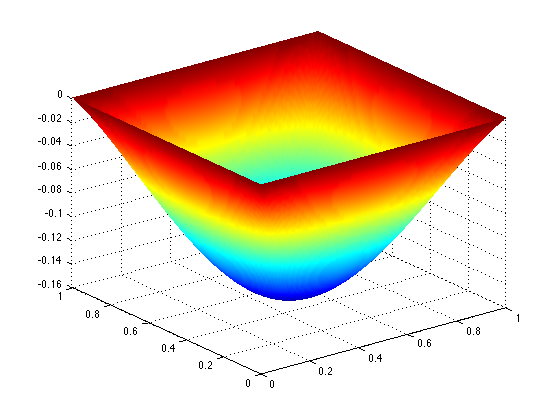} 
\end{center}
\caption{No known exact solution, $f(x,y)=1, g(x,y)=0, h=1/2^7,  \tir{\nu}=50$ with the compatible discretization \eqref{time-marching3}
} \label{fig1}
\end{figure}

\begin{figure}[tbp]
\begin{center}

\includegraphics[angle=0, height=4.5cm]{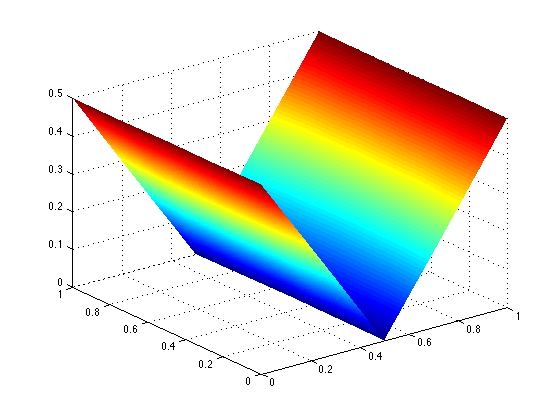}

\end{center}
\caption{$u(x,y)= |x-1/2|$ with $g(x,y)= |x-1/2|$ and  $f(x,y)=0$, $h=1/2^{2},  \tir{\nu} =5$ with the central discretization \eqref{time-marching}
} \label{fig2}
\end{figure}

\begin{table}
\begin{tabular}{c|cccccc} 
 \multicolumn{6}{c}{$h$}\\
 & $1/2^2$ &  $1/2^3$ &  $1/2^4$&  $1/2^5$ &  $1/2^6$ &  $1/2^7$ \\ \hline
Error & 3.91   $10^{-3}$ & 1.03 $10^{-3}$ & 2.66 $10^{-4}$ & 6.70 $10^{-5}$ & 1.68 $10^{-5}$ & 4.20 $10^{-6}$ \\
Newton & 0.0930 & 0.0354  & 0.1287   &  0.5796  & 3.5300  & 56.457  \\
$\tir{\nu}=4$  & 0.0334  & 0.0504 &  0.0679   & 0.1773   & 0.6721 &3.5300  \\
$\tir{\nu}=2.5$  & 0.0204 & 0.0359 &     &    &  &   \\
\end{tabular}
\caption{Computation times for Newton's method and time marching method \eqref{time-marching} for $u(x,y)=e^{(x^2+y^2)/2}$ }\label{tab3}
\end{table}

%% file: Standard-FD2-25b.tex
\begin{rem}
The convergence rate of the discretization is dictated by the estimate in Lemma \ref{mov-ball} and the expression of $\rho$. 
\end{rem}



\begin{rem} \label{meth1}
We point out that in dimension $d=2$, the method of this paper readily extends to the discretization
\begin{align} \label{m1h-2D}
\frac{1}{2} \div_h [ (\cof \mathcal{H}_2 u_{h} )^T D_h u_{h} ]  =r_h(f) \, \text{in} \, \Omega_h^0, u_{h} =r_h(g) \, \text{on} \, \partial \Omega_h.
\end{align}
In the above formulation we did not use the symmetric part of $\mathcal{H}_2 u_{h}$. This is because in dimension 2, $\cof \sym \mathcal{H}_2 v_h = \sym \cof \mathcal{H}_2 v_h$. And hence from the positive definiteness of $\sym \mathcal{H}_2 v_h$ we also get the positive definiteness of $\sym \cof \mathcal{H}_2 v_h$. The latter is an essential step in our approach for the proof of convergence of the time marching method for solving \eqref{m1h}.
See \eqref{contrac-p0} and \eqref{ess-step}. 
\end{rem}

\begin{rem} \label{meth2}
If we define $\mathcal{\hat{H}}_d(v_h)$ as $\tir{D}_h \tir{D}_h v_h$, we no longer have $\tr (\mathcal{\hat{H}}_d(v_h)) = \Delta_h (v_h)$. However it can be readily checked that the resulting discrete Hessian matrix is symmetric. Moreover, the proof of the divergence free row property of the cofactor matrix, \cite[p. 440]{Evans1998}, extends to the discrete case to yield $\div_h \cof \mathcal{\hat{H}}_d(v_h)=0$. Thus 
we obtain
$$
\frac{1}{d}  \div_h [(\cof \mathcal{\hat{H}}_d(v_h) ) D_h v_h] = \frac{1}{d} (\cof \mathcal{\hat{H}}_d(v_h) ):  \tir{D}_h D_h v_h, 
$$
and 
since $\det D^2 v = d^{-1} (\cof D^2 v):D^2 v$, we get
$$
 \frac{1}{d}  \div_h [(\cof \mathcal{\hat{H}}_d(r_h v) ) D_h r_h v](x) - \det D^2 v (x)  = O(h).
$$
As for Remark \ref{quad-cvg}, we also obtain a quadratic convergence rate in the maximum norm for the discretization
$$
\frac{1}{d} \div_h [ (\cof \mathcal{\hat{H}}_d u_{h} ) D_h u_{h} ]  =r_h(f) \, \text{in} \, \Omega_{h}^0, u_{h} =r_h(g) \, \text{on} \, \partial \Omega_{h}.
$$
The same remarks apply to the discretization
$$
\frac{1}{d} \div_h [ (\cof  \tir{\mathcal{H}}_d(u_{h}) ) D_h u_{h} ]  =r_h(f) \, \text{in} \, \Omega_{h}^0, u_{h} =r_h(g) \, \text{on} \, \partial \Omega_{h},
$$
and we recall that the matrix $ \tir{\mathcal{H}}_d(u_{h})$ was defined in section \ref{std-common}.
\end{rem}

\section{Appendix} \label{appendix}
\subsection{Proof of Lemma \ref{det-ap-lem}}
\begin{proof}
Let $x \in \Omega_h^0$ and let $v_h, w_h \in \mathcal{M}(\Omega_h)$. We have for $i=1,\ldots,d$
\begin{equation} \label{leibniz}
\partial_-^i (v_h w_h) (x) = v_h(x) \partial_-^i w_h(x) + w_h(x-h e_i) \partial_-^i v_h (x).
\end{equation}
This follows from
\begin{align*}
h \partial_-^i (v_h w_h) (x) & = v_h(x) w_h(x) - v_h(x -h e_i) w_h(x -h e_i) \\
& = v_h(x) (w_h (x) - w_h(x-h e_i)) + w_h(x-h e_i) (v_h(x) - v_h(x -h e_i)).
\end{align*}
For a vector field $v_h=(v_{i,h})_{i=1,\ldots,d}$, we define the ''translation'' matrix $\tau v_h$  by
$$
(\tau v_h)_{ij}(x) = v_{j,h}(x-h e_i),
$$
and for a matrix field $A_h=(A_{ij,h})_{i,j=1,\ldots,d}$, we define its gradient $\tir{D}_h A_h $ as the matrix field with components
$$
(\tir{D}_h A_h)_{ij} = \partial^i_- A_{ij,h}.
$$
Thus we have
\begin{equation}  \label{obs222}
\div_h (A_h v_h) (x)= (\tir{D}_h A_h) :(\tau v_h) + A_h(x): (\tir{D}_h v_h(x))^T.
\end{equation}
This follows from the Leibniz rule \eqref{leibniz}. Indeed we have
\begin{align*}
(A_h v_h)_i & = \sum_{j=1}^d A_{ij,h} v_{j,h} \\ 
\div_h (A_h v_h) (x) & = \sum_{i=1}^d \partial_-^i (A_h v_h)_i  (x)  =  \sum_{i,j=1}^d \partial_-^i (A_{ij,h}(x) v_{j,h} (x) )\\
& = \sum_{i,j=1}^d  A_{ij,h}(x) \partial_-^i v_{j,h} (x) + v_{j,h}(x-h e_i) \partial_-^i   A_{ij,h}(x),
\end{align*}
 which proves the claim.
 
We conclude that for a grid scalar function $v_h$
\begin{align*}
 \div_h [(\cof \mathcal{H}_d (v_h) )^T D_h v_h ](x)  & = [\tir{D}_h (\cof \mathcal{H}_d (v_h) )^T ] : (\tau D_h v_h) \\ 
 & \quad + (\cof \mathcal{H}_d (v_h (x)) )^T: (\tir{D}_h D_ h v_h(x))^T.
\end{align*}
Therefore
\begin{align} \label{div-xt}
\begin{split}
 \div_h [(\cof \mathcal{H}_d (v_h) )^T D_h v_h ](x)  & = [\tir{D}_h (\cof \mathcal{H}_d (v_h) )^T ] : (\tau D_h v_h) \\ 
 & \qquad \quad + (\cof \mathcal{H}_d (v_h (x)) ): \mathcal{H}_d (v_h (x)).
 \end{split}
\end{align}

To prove the convergence rates \eqref{ap03}--\eqref{ap04}, we first make some observations. 
For given $C^2(\tir{\Omega})$ functions $v_1$ and $v_2$, if $v_1 = v_{1,h} + O(h^2)$ and $v_2 = v_{2,h} + O(h^2)$, then
$$
v_1 v_2 = v_{1,h} v_{2,h} + O(h^2).
$$
Indeed for $x \in \mathbb{Z}_h$
\begin{align*}
v_1(x) v_2(x) & = (v_{1,h}(x) + O(h^2))( v_{2,h}(x) + O(h^2)) \\
& = v_{1,h}(x) v_{2,h}(x) +  v_{1,h}(x) O(h^2) +  v_{1,h}(x) O(h^2) + O(h^4).
\end{align*}
But since $v_{1,h} = v_1 + O(h^2)$, $v_{1,h}$ is uniformly bounded in $x$ and $h$. The same property holds for $v_{2,h}$. The claim follows. And it clearly extends to a finite product of $C^2(\tir{\Omega})$ functions $v_i$ such that
$v_i = v_{i,h} + O(h^2), i=1,\ldots,d$. Similarly for $v_1 = v_{1,h} + O(h)$ and $v_2 = v_{2,h} + O(h)$, then
$$
v_1 v_2 = v_{1,h} v_{2,h} + O(h).
$$

Now, we have
\begin{align} \label{div-ap-p1}
\frac{1}{d} (\cof D^2 v): D^2 v - \frac{1}{d} (\cof \mathcal{H}_d(r_h v) ): \mathcal{H}_d(r_h v) =O(h),
\end{align}
since each product in the second term on the left is at least a linear approximation of a corresponding product in the first term on the left. 

On the other hand, each entry of $\tau D_h v_h$ is a linear approximation of a directional derivative of $v$. Similarly, each entry of $\tir{D}_h (\cof \mathcal{H}_d (v_h) )^T$ is a linear approximation of the corresponding third order derivative of $v$. Thus
\begin{align} \label{div-ap-p2}
[\tir{D}_h (\cof \mathcal{H}_d (v_h) )^T ] : (\tau D_h v_h)  -( \div \cof D^2 v) \cdot D v = O(h). 
\end{align}
Finally by \eqref{div-xt}, we have
\begin{align*}
\det D^2 v(x) & -  \frac{1}{d}  \div_h [(\cof \mathcal{H}_d (r_h v) )^T D_h r_h v](x)  = \frac{1}{d} (\cof D^2 v (x) ): D^2 v(x) \\
& \qquad - \frac{1}{d} (\cof \mathcal{H}_d(r_h v (x)) ): \mathcal{H}_d(r_h v (x)) \\
 & \qquad \qquad  -  \frac{1}{d} [\tir{D}_h (\cof \mathcal{H}_d (v_h) )^T ] : (\tau D_h r_h v),
\end{align*}
and \eqref{ap03} follows from \eqref{div-ap-p1} and \eqref{div-ap-p2}.

One proves  \eqref{ap04} similarly to  \eqref{ap03} since  each entry of $\cof \sym \mathcal{H}_d(r_h v)$ is also a linear approximation of the corresponding entry in $\cof D^2 v$.  
\end{proof}
